\documentclass[10pt]{article}
\usepackage{amsmath,amssymb, amsthm}
\usepackage[margin=1in]{geometry}
\usepackage{graphics}
\usepackage{tikz}
\usepackage{hyperref}
\hypersetup{
colorlinks = true,
citecolor = red
}
\newtheorem{theorem}{Theorem}
\newtheorem{lemma}{Lemma}
\newtheorem{prop}{Proposition}
\newtheorem{corollary}{Corollary}

\theoremstyle{definition}
\newtheorem{ex}{Example}
\newtheorem{defn}{Definition}
\newtheorem{remark}{Remark}

\def\multiset#1#2{\ensuremath{\left(\kern-.3em\left(\genfrac{}{}{0pt}{}{#1}{#2}\right)\kern-.3em\right)}}

\begin{document}

\title {Some Aspects of Higher Continued Fractions}
\author {Etan Basser, Nicholas Ovenhouse, and Anuj Sakarda}
\date{}

\maketitle

\abstract{
    We investigate some properties of the higher continued fractions defined recently by Musiker, Ovenhouse, Schiffler, and Zhang. We prove that the maps defining the higher continued fractions are increasing continuous functions on the positive real numbers. We also investigate some asymptotics of these maps.
}

\section{Introduction}

A new type of higher continued fractions was introduced in \cite{mosz}. The motivation for this definition came from combinatorial
considerations regarding the enumeration of higher dimer covers of certain planar graphs called \emph{snake graphs}, and an attempt to
generalize known relations between dimer covers of snake graphs and ordinary continued fractions. We will now describe some of these motivations in more detail, and then define the higher continued fractions.

In the theory of cluster algebras, there is the celebrated \emph{Laurent phenomenon}, which says that all cluster variables are Laurent polynomials when expressed in terms of the initial cluster variables \cite{fz_02}. For cluster algebras coming from triangulated surfaces, these Laurent polynomial expressions are known to be weighted sums
of perfect matchings (or \emph{dimer covers}) of planar graphs called \emph{snake graphs} \cite{ms_10} \cite{msw_11} \cite{msw_13}.

The geometric realization of these surface cluster algebras is the \emph{decorated Teichm\"{u}ller space} of the surface. The cluster variables are Penner's ``\emph{$\lambda$-lengths}''\cite{penner_book}, and the mutations are the hyperbolic version of Ptolemy's relation. The interpretation in terms of cluster algebras was explained in \cite{fomin_thurston} and \cite{gsv_weil_petersson}. The \emph{decorated super Teichm\"{u}ller space} of Penner and Zeitlin \cite{pz_19} was interpreted through the lens of cluster algebras in \cite{moz_21}, \cite{moz_22}, and \cite{moz_23}, where the authors proved a Laurent phenomenon for these super cluster structures by giving explicit Laurent polynomial expressions for super $\lambda$-lengths as weighted sums over double-dimer covers of snake graphs. This was one of the main original motivations for finding enumerative formulas for double (and higher) dimer covers in \cite{mosz}.

On the other hand, there is also an established relationship between these surface-type cluster algebras, their perfect matching Laurent formulas, and continued fractions \cite{cs_18}. In particular, a rational number defines a snake graph via its continued fraction expansion, and one of the main results of \cite{cs_18} is that the numerator of the rational number counts the perfect matchings of the snake graph. The authors of \cite{mosz} made an analogy between this known relationship and their new enumerative formulas for higher dimer covers to define their \emph{higher continued fractions}.

Before defining these higher continued fractions, we will briefly review the classical case. It is well-known that the convergents of
a continued fraction can be computed by multiplying certain $2 \times 2$ matrices. In particular, if $\frac{p}{q} = [a_1,\dots,a_n]$
and $\frac{p'}{q'} = [a_1,\dots,a_{n-1}]$, then
\[ 
    \begin{pmatrix} a_1 & 1 \\ 1 & 0 \end{pmatrix}
    \begin{pmatrix} a_2 & 1 \\ 1 & 0 \end{pmatrix}
    \cdots
    \begin{pmatrix} a_n & 1 \\ 1 & 0 \end{pmatrix}
    =
    \begin{pmatrix} p & p' \\ q & q' \end{pmatrix}
\]
Another representation of this matrix product involves the three matrices
\[ 
    R = \begin{pmatrix} 1&1 \\ 0&1 \end{pmatrix}, \quad \quad
    L = \begin{pmatrix} 1&0 \\ 1&1 \end{pmatrix}, \quad \quad
    W = \begin{pmatrix} 0&1 \\ 1&0 \end{pmatrix}.
\]
If $n$ is even, then the matrix product above is equal to $R^{a_1}L^{a_2}R^{a_3}L^{a_4}\cdots L^{a_n}$,
and if $n$ is odd, then it is equal to $R^{a_1}L^{a_2} \cdots R^{a_n}W$. The most straightforward definition
of the higher continued fractions from \cite{mosz} is by replacing these matrix products with certain $(m+1) \times (m+1)$ counterparts,
where $m \geq 1$. Let $R_m$ and $L_m$ be the upper and lower triangular matrices with all 1's, $W_m$ the permutation matrix
for the longest permutation in the symmetric group $S_{m+1}$, and for an integer $a$ a matrix $\Lambda_m(a)$:
\[ 
    R_m = \begin{pmatrix} 
    1 & 1 & \cdots & 1 \\
    0 & 1 & \cdots & 1 \\
    0 & 0 & \ddots & \vdots \\
    0 & 0 & \cdots & 1
    \end{pmatrix}, \quad \quad
    L_m = \begin{pmatrix} 
    1 & \cdots & 0 & 0 \\
    \vdots & \ddots & 0 & 0 \\
    1 & \cdots & 1 & 0 \\
    1 & \cdots & 1 & 1
    \end{pmatrix}, \quad \quad
    W_m = \begin{pmatrix}
    0      & 0      & 0 & 1 \\
    0      & 0      & 1 & 0 \\
    \vdots & \hspace{-2ex}\rotatebox{90}{$\ddots$} & 0 & 0 \\
    1      & \cdots & 0 & 0
    \end{pmatrix}, 
\]
\[
    \Lambda_m(a) = R_m^a W_m = W_m L_m^a = \begin{pmatrix}
        \multiset{a}{m}   & \multiset{a}{m-1} & \cdots & a &  1 \\
        \multiset{a}{m-1} & \multiset{a}{m-2} & \cdots & 1 & 0 \\
        \vdots            & \vdots            & \hspace{-2ex}\rotatebox{90}{$\ddots$} & 0 & 0 \\
        a                 & 1                 & 0      & \hspace{-2ex}\rotatebox{90}{$\ddots$} & 0 \\
        1                 & 0                 & 0      & 0 & 0
    \end{pmatrix}
\]
Here, the notation $\multiset{n}{k}$ (``$n$ \emph{multichoose} $k$'') denotes the binomial coefficient $\binom{n+k-1}{k}$,
which is the number of multisets of size $k$ whose elements come from the set $\{1,2,\dots,n\}$.

A family of maps $r_{i,m} \colon \Bbb{R}_{\geq 1} \to \Bbb{R}_{\geq 1}$ was defined in \cite{mosz} as follows. 

\begin {defn} \label{def:CF}
For a rational number $x \geq 1$
with continued fraction $x = [a_1,\dots,a_n]$, let $M$ be the following matrix product:
\[ M = \Lambda_m(a_1) \Lambda_m(a_2) \cdots \Lambda_m(a_n). \] 
We define $\mathrm{CF}_m(x)$ to be the first column of $M$, normalized so that the last entry is 1 (i.e. divided by $M_{m+1,1}$). 
The entries of this vector are denoted by $r_{i,m}(x)$:
\[ \mathrm{CF}_m(x) = \left( r_{m,m}(x), r_{m-1,m}(x), \dots, r_{2,m}(x), r_{1,m}(x), 1 \right)^\top \]
\end {defn}

\begin {remark}
It is also sometimes convenient to think of $\mathrm{CF}_m(x)$ as the homogeneous coordinates of a point in the projective
space $\Bbb{P}^m$, and the $r_{i,m}(x)$ are the affine coordinates in the chart where the last coordinate is non-zero. This allows one to deal with the matrix $M$ directly without dividing by the bottom-left entry.
\end {remark}

\begin {ex}
    For $m=1$, we always have $r_{1,1}(x) = x$, and $\mathrm{CF}_1(x) = (x,1)^\top$.
\end {ex}

\begin {ex}
    The continued fraction for $\frac{12}{7}$ is $[1,1,2,2]$. The vector $\mathrm{CF}_2(12/7)$ is therefore computed using the matrix product 
    \[ \Lambda_2(1)^2\Lambda_2(2)^2 = \begin{pmatrix} 61 & 36 & 14 \\ 47 & 28 & 11 \\ 25 & 15 & 6 \end{pmatrix} \]
    Therefore the values of $r_{i,2}(12/7)$ are given by
    \[ \mathrm{CF}_2\left( \frac{12}{7} \right) = \left( r_{2,2} \left( \frac{12}{7} \right), \, r_{1,2} \left( \frac{12}{7} \right), \, r_{0,2} \left( \frac{12}{7} \right) \right) = \left( \frac{61}{25}, \, \frac{47}{25}, \, 1 \right)^\top \]
\end {ex}

One of the main results in \cite{mosz} was that this definition extends to real (rather than rational) values of $x$.

\medskip

\begin {theorem} (\cite{mosz})
If $x = [a_1,a_2,\dots]$ is the continued fraction for an irrational number, and if $x_n = [a_1,\dots,a_n]$ are its rational convergents, then the sequence $r_{i,m}(x_n)$ converges.
\end {theorem}

\medskip

\begin {defn}
If $x = [a_1,a_2,\dots]$ is irrational, we define $r_{i,m}(x)$ as the limit $\lim_{n \to \infty} r_{i,m}(x_n)$.
\end {defn}

\medskip

\begin {ex}
    Let $f_n$ be the sequence of Fibonacci numbers. It is well-known that the sequence of ratios $x_n = \frac{f_n}{f_{n-1}}$ have continued fractions with all $1$'s: $x_n = [1,1,1,\dots,1]$, and that $\lim_{n \to \infty} x_n = \varphi = \frac{1}{2}(1+\sqrt{5})$, the golden ratio. The sequence $r_{2,2}(x_n)$ is the following sequence of rationals (and their approximate decimal values):
    \[ 3, \quad 2, \quad \frac{14}{6}, \quad \frac{31}{14}, \quad \frac{70}{31}, \quad \frac{157}{70}, \quad \dots \]
    \[ 3, \quad 2, \quad 2.333, \quad 2.214, \quad 2.258, \quad 2.243, \quad \dots \]
    This sequence converges to
    \[ r_{2,2}(\varphi) = 4\cos^2(\pi/7)-1 \approx 2.24698 \]
\end {ex}

In \cite{mosz}, it was conjectured based on computer experiments that each $r_{i,m}$ is an increasing and continuous function.
The two main results of this work (Theorem \ref{thm1} and Theorem \ref{thm2}) establish the monotonicity and continuity for positive numbers.
Another open problem is to find an inverse to the $r_{i,m}$ maps (if one exists), giving some generalization of the usual algorithm
for computing a continued fraction representation. While we do not attempt to find an inverse in the present work, we note that
our results that $r_{i,m}$ is strictly increasing and continuous imply that it is a bijection $[0,\infty) \to [0,\infty)$, and so we can conclude that such an
inverse does exist on this domain.

\section{Monotonicity}

In this section, we will prove that the maps $r_{i,m}$ are strictly increasing.

\bigskip

\begin{lemma} \label{lem:first_lemma}
    For positive integers $a,m,i,j,k$, with $j < k$ and $1 \leq i,j,k \leq m+1$, we have
    \[ \multiset{a}{m+2-(i+k)} \multiset{a}{m+1-(i+j)} \geq \multiset{a}{m+2-(i+j)}\multiset{a}{m+1-(i+k)}. \]
    If $a > 1$ and $j,k \leq m+2-i$, then the inequality is strict.
\end{lemma}

\begin{proof}
For ease of notation, let us make the substitution $r=m+2-i$. Therefore, we need to show that \[ \multiset{a}{r-k} \multiset{a}{r-j-1} \geq \multiset{a}{r-j}\multiset{a}{r-k-1} \]

First let us get some degenerate cases out of the way. If either of $j$ or $k$ is greater than $r$, then the inequality simply becomes $0 \geq 0$. In the case when $k = r$, the inequality simply says $\multiset{a}{r-1-j} \geq 0$. For the rest of the proof, we assume that $j < k \leq r-1$.

Note that $\frac{\multiset{a}{r-k}}{\multiset{a}{r-k-1}}=\frac{\binom{a+r-k-1}{r-k}}{\binom{a+r-k-2}{r-k-1}} = \frac{a+r-k-1}{r-k} = 1 + \frac{a-1}{r-k}$. Similarly, $\frac{\multiset{a}{r-j}}{\multiset{a}{r-j-1}}= 1 + \frac{a-1}{r-j}$. Since $j < k$, this means that  $\frac{\multiset{a}{r-k}}{\multiset{a}{r-k-1}}\geq{\frac{\multiset{a}{r-j}}{\multiset{a}{r-j-1}}}$, with equality only if $a=1$. Finally cross-multiplying gives the desired result. 

\end{proof}

\begin{lemma} \label{lemma1}
    Suppose we have two vectors $X = (x_1,x_2,\dots,x_{m+1})$ and $Y = (y_1,y_2,\dots,y_{m+1})$ with all positive entries, such that $\frac{x_i}{x_{i+1}} \geq \frac{y_i}{y_{i+1}}>{0}$ for all $ 1 \leq i \leq m$. 
    Also let $a \in \Bbb{N}$ and define $X' = \Lambda_m(a) X$ and $Y' = \Lambda_m(a) Y$.
    Then $\frac{y'_i}{y'_{i+1}} \geq \frac{x'_i}{x'_{i+1}}>{0}$ for all $1 \leq i \leq m$.
    In other words, multiplication by a $\Lambda$-matrix reverses these inequalities.

    Furthermore, if the hypothesized inequality $\frac{x_1}{x_{2}} > \frac{y_1}{y_{2}}$ is strict for $i=1$, 
    then $\frac{y'_i}{y'_{i+1}} > \frac{x'_i}{x'_{i+1}}$ is strict for all $1 \leq i \leq m$.

\end{lemma}
\begin{proof}
    Note that for all $i$, we have 
    \[ x'_i = (\Lambda_m(a)X)_i = {\sum_{j=1}^{m+1} \multiset{a}{m+2-(i+j)}x_j} \] 
    and similarly for $Y'$. Thus, the desired inequality $\frac{y'_i}{y'_{i+1}} \geq \frac{x'_i}{x'_{i+1}}$ is equivalent to
    \[ \frac{\sum_{j=1}^{m+1} \multiset{a}{m+2-(i+j)}y_j}{\sum_{j=1}^{m+1} \multiset{a}{m+1-(i+j)}y_j} \geq \frac{\sum_{j=1}^{m+1} \multiset{a}{m+2-(i+j)}x_j}{\sum_{j=1}^{m+1} \multiset{a}{m+1-(i+j)}x_j} \]
    or equivalently
    \[ \sum_{j,k=1}^{m+1} \multiset{a}{m+2-(i+j)}y_j\multiset{a}{m+1-(i+k)}x_k \geq \sum_{j,k=1}^{m+1} \multiset{a}{m+2-(i+j)}x_j\multiset{a}{m+1-(i+k)}y_k \]
    We will prove a stronger statement; namely that for all $1 \leq j \leq k \leq m+1$, we have 
    \begin{multline*}
      \multiset{a}{m+2-(i+j)}y_j\multiset{a}{m+1-(i+k)}x_k + \multiset{a}{m+2-(i+k)}y_k\multiset{a}{m+1-(i+j)}x_j \geq \\ 
      \multiset{a}{m+2-(i+j)}x_j\multiset{a}{m+1-(i+k)}y_k + \multiset{a}{m+2-(i+k)}x_k\multiset{a}{m+1-(i+j)}y_j.
    \end{multline*}
    Assuming this, then by summing over all pairs $j \leq k$, we will obtain the desired inequality above.
    By rearranging, we equivalently need to show that 
    \[ \multiset{a}{m+2-(i+k)}\multiset{a}{m+1-(i+j)}(x_jy_k-x_ky_j) \geq \multiset{a}{m+2-(i+j)}\multiset{a}{m+1-(i+k)}(x_jy_k-x_ky_j) \]
    Since we assume that $\frac{x_i}{x_{i+1}}\geq\frac{y_i}{y_{i+1}}$ for all $i$, and because $\frac{x_j}{x_k} = \frac{x_j}{x_{j+1}} \frac{x_{j+1}}{x_{j+2}} \cdots \frac{x_{k-1}}{x_k}$,
    we find that $\frac{x_j}{x_k} \geq \frac{y_j}{y_k}$ whenever $j < k$.
    Since $x_jy_k - x_ky_j \geq 0$, we can cancel this factor from both sides, and this becomes the inequality from Lemma \ref{lem:first_lemma}.

Now, additionally assume that  $\frac{x_1}{x_{2}} > \frac{y_1}{y_{2}}$. To show strictness, since we are summing over separate inequalities, it suffices to show there exist some $1\leq{j}<k\leq{m+1}$ such that
\[ \multiset{a}{m+2-(i+k)}\multiset{a}{m+1-(i+j)}(x_jy_k-x_ky_j) > \multiset{a}{m+2-(i+j)}\multiset{a}{m+1-(i+k)}(x_jy_k-x_ky_j) \] 
Indeed, take $k=m+2-i$ and $j=1$. Then 
\[ \multiset{a}{m+2-(i+k)}\multiset{a}{m+1-(i+j)} = \multiset{a}{0} \multiset{a}{m-i} = \multiset{a}{m-i} \geq 1, \quad \text{and} \] 
\[ \multiset{a}{m+2-(i+j)}\multiset{a}{m+1-(i+k)} = \multiset{a}{m+1-i} \multiset{a}{-1} = 0. \] 
Thus it suffices to show that $x_jy_k-x_ky_j>0$ -- i.e., $\frac{x_j}{x_k} > \frac{y_j}{j_k}$. But this holds because by assumption $\frac{x_1}{x_{2}} > \frac{y_1}{y_{2}}$ and $\frac{x_i}{x_{i+1}} \geq \frac{y_i}{y_{i+1}}>{0}$ for all $ 1 \leq i \leq m$.

\end{proof}

\begin{corollary}
\label{cor1}
    Let $a_1,\dots,a_n$ be a sequence of positive integers, and let $M = \Lambda_m(a_1)\Lambda_m(a_2) \cdots \Lambda_m(a_n)$. 
    Take $X$ and $Y$ as in Lemma \ref{lemma1}, and let $X' = MX$ and $Y' = MY$.
    Then for all $i$, we have
    $\frac{x'_i}{x'_{i+1}} \geq \frac{y'_i}{y'_{i+1}}$ if $n$ is even, and $\frac{x'_i}{x'_{i+1}} \leq \frac{y'_i}{y'_{i+1}}$ if $n$ is odd. Furthermore, if in addition $\frac{x_1}{x_{2}} > \frac{y_1}{y_{2}}$, then the strict inequalities hold.
\end{corollary}
\begin{proof}
    When $n=0$, the inequality holds by assumption on $X$ and $Y$. The inductive step is given by Lemma \ref{lemma1}, which says that the inequality
    is reversed with each multiplication by another $\Lambda_m(a)$ matrix.
\end{proof}

\begin{lemma} \label{lemma2}
    Let $x \geq 1$ be a real number, and let $n = \lfloor x \rfloor$ be the integer part. For all $m \geq 1$ and $1 \leq i \leq m$, we have 

$$\begin{cases}
\displaystyle \frac{r_{i,m}(x)}{r_{i-1,m}(x)} = \frac{n-1+i}{i}\hspace{0.3cm} & \text{ if } x = n \in{\mathbb{N}}\\[2ex]
  \displaystyle \frac{n-1+i}{i} < \frac{r_{i,m}(x)}{r_{i-1,m}(x)} < \frac{n+i}{i}\hspace{0.3cm} & \text{ otherwise } 
\end{cases}$$
\end{lemma}

\begin{proof}
First, when $x = n$ is an integer, it is clear that $r_{i,m}(n) = \multiset{n}{i}$, and so the ratio is $\frac{r_{i,m}(n)}{r_{i-1,m}(n)} = \frac{n-1+i}{i}$.

Let $[a_1,a_2,\dots]$ be the continued fraction for $x$, and let $x' = [a_2,a_3,\dots]$.
Then $$\mathrm{CF}_m(x) = \frac{1}{r_{m,m}(x')}\Lambda_m(a_1) \mathrm{CF}_m(x'),$$and therefore the ratio we are interested in is
$\frac{\left( \Lambda_m(a_1) \mathrm{CF}_m(x') \right)_{m+1-i}}{\left( \Lambda_m(a_1) \mathrm{CF}_m(x') \right)_{m+2-i}}$.
We prove the left inequality first. 
\begin{align*}
    \frac{r_{i,m}(x)}{r_{i-1,m}(x)} 
    &= \frac{(\Lambda_m(a_1) \mathrm{CF}_m(x'))_{m+1-i}}{(\Lambda_m(a_1)\mathrm{CF}_m(x'))_{m+2-i}} \\
    &= \frac{\sum_{k=0}^i \multiset{a_1}{k}r_{m-i+k,m}(x')}{\sum_{k=0}^{i-1}\multiset{a_1}{k}r_{m-i+1+k,m}(x')} \\
    &= \frac{r_{m-i,m}(x') + \sum_{k=1}^i \multiset{a_1}{k} r_{m-i+k,m}(x')}{\sum_{k=1}^i \multiset{a_1}{k-1} r_{m-i+k,m}(x')} \\
    &> \frac{\sum_{k=1}^i \multiset{a_1}{k} r_{m-i+k,m}(x')}{\sum_{k=1}^i \multiset{a_1}{k-1} r_{m-i+k,m}(x')} \\
    &= \frac{\sum_{k=1}^i \multiset{a_1}{k-1} \frac{a_1+k-1}{k} r_{m-i+k,m}(x')}{\sum_{k=1}^i \multiset{a_1}{k-1} r_{m-i+k,m}(x')} \\
    &\geq \frac{\sum_{k=1}^i \multiset{a_1}{k-1} \frac{a_1+i-1}{i} r_{m-i+k,m}(x')}{\sum_{k=1}^i \multiset{a_1}{k-1} r_{m-i+k,m}(x')} \\
    &= \frac{a_1+i-1}{i}
\end{align*}

Now, we treat the right inequality. Note that by the left inequality, $\frac{r_{m,m}(x)}{r_{m-1,m}(x)} > \frac{\lfloor{x}\rfloor+m-1}{m}\geq{1}$. 

Consider the vector $V = (1,1,\dots,1) \in \Bbb{R}^{m+1}$ consisting of all 1's. 
Thus, the hypotheses of the strict inequality in Lemma \ref{lemma1} are satisfied because $r_{m,m}(x) > r_{m-1,m}(x) \geq \cdots \geq r_{1,m}(x) \geq r_{0,m}(x)$. 
So, for any $x\notin{\mathbb{N}}$, we have 

\begin{align*}
    \frac{r_{i,m}(x)}{r_{i-1,m}(x)} 
    &= \frac{(\Lambda_m(a_1) \mathrm{CF}_m(x'))_{m+1-i}}{(\Lambda_m(a_1)\mathrm{CF}_m(x'))_{m+2-i}} \\
    &< \frac{(\Lambda_m(a_1)V)_{m+1-i}}{(\Lambda_m(a_1)V)_{m+2-i}} \tag{Lemma \ref{lemma1}}\\
    &= \frac{\sum_{j=0}^{i} \multiset{a_1}{j}}{\sum_{j=0}^{i-1} \multiset{a_1}{j}}\\
    &= \frac{\multiset{a_1+1}{i}}{\multiset{a_1+1}{i-1}}\\
    &= \frac{a_1+i}{i},
\end{align*}
where the penultimate equality holds by the multichoose Hockey Stick Identity, which states that $$\sum_{r=0}^{s} \multiset{a}{r} = \multiset{a+1}{s}.$$

\end{proof}

This statement leads to the following bounds.

\begin{corollary} \label{cor:bounds}
    Let $x \in \Bbb{R}$ (with $x \geq 1$) with continued fraction $x = [a_1,a_2,\dots]$. For all $1 \leq j<k \leq m$, we have 
    \[ \frac{\multiset{a_1}{k}}{\multiset{a_1}{j}} \leq \frac{r_{k,m}(x)}{r_{j,m}(x)} < \frac{\multiset{a_1+1}{k}}{\multiset{a_1+1}{j}},\] with equality holding only in the case when $x$ is an integer.
    In particular, taking $j=0$, we have $\multiset{a_1}{k} \leq r_{k,m}(x) < \multiset{a_1+1}{k}$.
\end{corollary}

\begin{proof}
    By telescopically multiplying the inequalities in Lemma \ref{lemma2}, we see 
    \[ \prod_{i=j+1}^{k} \frac{a_1-1+i}{i} \leq \prod_{i=j+1}^{k} \frac{r_{i,m}(x)}{r_{i-1,m}(x)} < \prod_{i=j+1}^{k} \frac{a_1+i}{i}. \] 

    After many cancellations, the middle product is simply $\frac{r_{k,m}(x)}{r_{j,m}(x)}$. The outer products are the appropriate ratios of binomial coefficients.
\end{proof}

\begin{corollary} \label{cor2}
    Let $x = [a_1,a_2,\dots]$ and $y=[b_1,b_2,\dots]$ be real numbers, with $a_1 > b_1$ (that is, $\lfloor x \rfloor > \lfloor y \rfloor$). Then for all $j$, we have
    \[ \frac{r_{j,m}(x)}{r_{j-1,m}(x)} > \frac{r_{j,m}(y)}{r_{j-1,m}(y)}. \]

\end{corollary}
\begin{proof}
    By Lemma \ref{lemma2}, we have 
    \[ \frac{r_{j,m}(x)}{r_{j-1,m}(x)} \geq \frac{a_1-1+j}{j} \geq \frac{b_1+j}{j} > \frac{r_{j,m}(y)}{r_{j-1,m}(y)} \]
\end{proof}

\begin {lemma} \label{lem4}
    The maps $\frac{r_{i,m}}{r_{i-1,m}}$ are strictly increasing on the interval $[1,\infty)$. That is, for $1 \leq x < y$, we have $\frac{r_{i,m}(x)}{r_{i-1,m}(x)} < \frac{r_{i,m}(y)}{r_{i-1,m}(y)}$.
\end {lemma}
\begin {proof}
    Let $x$ and $y$ have continued fractions $x = [a_1,a_2,\dots]$ and $y = [b_1,b_2,\dots]$. Suppose $a_i=b_i$ for all $i<k$ and $a_k \neq b_k$.
    That is, suppose the continued fractions for $x$ and $y$ agree up to (but not including) the $k^\mathrm{th}$ position.
    The case $k=1$ was proved in Corollary \ref{cor2}.

    In general, it is known that for $k$ odd, $x < y$ if and only if $a_k < b_k$, and for $k$ even, $x < y$ if and only if $a_k > b_k$.
    Let's consider the case that $k$ is odd (the even case is similar). Let $x' = [a_k, a_{k+1}, \dots]$ and $y' = [b_k, b_{k+1}, \dots]$.
    By Corollary \ref{cor2}, since $a_k < b_k$, we have $\frac{r_{j,m}(x')}{r_{j-1,m}(x')} < \frac{r_{j,m}(y')}{r_{j-1,m}(y')}$ for all $j$. Since $\mathrm{CF}_m(x) = \Lambda_m(a_1)\Lambda_m(a_2) \cdots \Lambda_m(a_{k-1}) \mathrm{CF}_m(x')$ (up to a scalar multiple),
    and since we assume $k$ is odd, the result follows by Corollary \ref{cor1}.

    There is still the case where $a_1,a_2,\dots,a_k$ is a substring of $b_1,b_2,\dots$. That is, suppose that $x=[a_1,\dots,a_k]$, $y=[b_1,b_2,\dots]$, and that $a_i = b_i$ for $1 \leq i \leq k$. Again, let $y' = [b_k,b_{k+1},\dots]$. Then $x < y$ if $k$ is odd, and $x > y$ if $k$ is even. By Corollary \ref{cor:bounds} (or Lemma \ref{lemma2}), $\frac{r_{i,m}(y')}{r_{i-1,m}(y')} > \frac{r_{i,m}(b_k)}{r_{i-1,m}(b_k)}$. We obtain $x$ and $y$ from $[b_k]$ and $[b_k, b_{k+1},\dots]$ via multiplication by $\Lambda(a_1) \cdots \Lambda(a_{k-1})$. As in the above argument, the result follows by Corollary \ref{cor1}.
\end {proof}

\begin {theorem} \label {thm1} 
The maps $r_{i,m}(x)$ are strictly increasing on $[1,\infty)$. 
In addition, the maps $\frac{r_{m,m}(x)}{r_{i,m}(x)}$ are strictly increasing on $[1,\infty)$.
\end {theorem}

\begin {proof}
Since $r_{0,m}(x) = 1$, then $r_{i,m}(x) = \frac{r_{i,m}(x)}{r_{i-1,m}(x)} \cdots \frac{r_{1,m}(x)}{r_{0,m}(x)}$. The result follows by multiplying the inequalities from Lemma \ref{lem4}.

The corresponding proof for $\frac{r_{m,m}}{r_{i,m}}$ follows by an analogous calculation.
\end {proof}

\section{Continuity}

In this section, we will prove that the maps $r_{i,m}(x)$ are continuous.

\begin{lemma}
\label{lemma5}
Let $z$ be a positive real number with continued fraction $z=[c_1,c_2,c_3\dots]$. Let $z_1,z_2,\dots$ be a sequence converging to $z$, whose terms have continued fractions $z_i = [d_{i,1},d_{i,2},d_{i,3}\dots]$. 

\begin {itemize}
    \item[(a)] Suppose $z$ is irrational. Then for all $k\geq{1}$, there is some $N$ such that $d_{i,j} = c_j$ for $i>N$ and $j \leq k$.
    \item[(b)] Suppose $z$ is rational, with finite continued fraction $z=[c_1,c_2,\dots,c_m]$, and $c_m > 1$. Then there is some $N$ such that when $i>N$,
               $d_{i,j} = c_j$ for $j \leq m-1$, and $d_{i,m}$ is either $c_m$ or $c_m-1$.
\end {itemize}
\end{lemma}

\begin{proof}
We'll prove the assertion for irrational values first. 

We induct on $k$. The base case $k=1$ holds trivially. To see why, since $z$ is irrational (in particular not an integer) and the sequence of $z_i$ converges to $z$, we must have that $d_{i,1}=\lfloor{z_i}\rfloor$ converges to $c_1=\lfloor{z}\rfloor$. Therefore, we eventually have $d_{i,1}=c_1$. Now assume the statement holds for some natural number $k$. By the base case, we know that eventually $d_{i,1}=c_1$. Furthermore, since $z$ is irrational eventually $z_i>c_1$. Therefore the sequence given by $w_i=\frac{1}{z_i-c_1}$ is eventually well-defined and clearly converges to $\frac{1}{z-c_1}$. Thus by the inductive hypothesis, eventually the first $k$ terms of the continued fraction representation of $w_i$ agree with the first $k$ terms of the continued fraction representation of $\frac{1}{z-c_1}$. Thus, by construction of continued fractions, eventually the first $k+1$ terms of the continued fraction representation of $z_i$ agree with the first $k+1$ terms of the continued fraction representation of $z$.

Now we treat the rational case. We'll employ a very similar inductive argument. This time we'll induct on $m$. The base case $m=1$ (i.e., $z=c_1$) holds trivially because eventually $c_1-1<z_i<c_1+1$, meaning that $\lfloor{z_i}\rfloor=d_{i,1}\in\{c_1-1,c_1\}$. Now assume the statement holds for some natural number $m$. Consider some $z'$ with continued fraction representation $[c_1,c_2,c_3\dots,c_{m+1}]$ (and $c_{m+1} > 1$) and a sequence $z_i'$ converging to $z'$. Since $z$ is not an integer (as $m>1$), eventually the first term of the continued fraction representations of the $z_i'$ all are $c_1$. Now proceed in the same manner as in the argument for the irrational case. 
\end{proof}

\begin{lemma}
\label{lemma6}
For all $m>1$ and $1\leq i \leq m$, the functions $r_{i,m}$ are continuous at the natural numbers. 
\end{lemma}
\begin{proof}

Let $a \in \Bbb{N}$ be a natural number. Consider the sequence of continued fractions given by $[a-1,1,b] = a-\frac{1}{b+1}$ and $[a,b]=a+\frac{1}{b}$. Note that both sequences converge to $a$ as $b \to \infty$, and that $[a-1,1,b] < a < [a,b]$. 

Now, let $x_i$ be a sequence converging to $a$. For large enough $k$, we will have $a - \frac{1}{2} < x_k < a + \frac{1}{2}$.
Remove finitely many terms from the beginning of the sequence so that this is true for all $k$.
Then there is a non-decreasing sequence $b_1 \leq b_2 \leq b_3 \leq \dots$ such that $\lim_{k \to \infty} b_k = \infty$ and 
$[a-1,1,b_k] \leq x_k \leq [a,b_k]$ for all $k$. Because the $r_{i,m}$ are monotone (by Theorem~\ref{thm1}), 
we have that $r_{im} \left( a - \frac{1}{b_k+1} \right) \leq r_{im}(x_k) \leq r_{im} \left( a + \frac{1}{b_k} \right)$. 
Therefore it suffices to show that 
\[ \lim_{b \to \infty} r_{i,m}([a,b]) = \lim_{b \to \infty} r_{i,m}([a-1,1,b]) = r_{i,m}(a). \] 

By a direct calculation, 
\[ r_{i,m}([a,b]) = \sum_{j=0}^{i} \multiset{a}{i-j} \frac{\multiset{b}{m-j}}{\multiset{b}{m}}.\]
Note that $\lim_{b \to \infty} \frac{\multiset{b}{m-j}}{\multiset{b}{m}} = 0$ for $j>0$. Therefore $\lim_{b \to \infty} r_{i,m}([a,b])$ is precisely $\multiset{a}{i}=r_{i,m}(a)$.

Now by a similar direct calculation, we see that 
 
\[ r_{i,m}([a-1,1,b]) = \sum_{j=0}^{i} \multiset{a-1}{i-j} \frac{\sum_{k=j}^m \multiset{b}{k}}{\sum_{k=0}^m \multiset{b}{k}}.\]

For each $j$, the ratio $\frac{\sum_{k=j}^m \multiset{b}{K}}{\sum_{k=0}^m \multiset{b}{k}}$ approaches $1$ as $b \to \infty$ (since both numerator and denominator are polynomials in $b$ with the same leading term), and we are left with just $\sum_{j=0}^{i} \multiset{a-1}{i-j} = \multiset{a}{i} = r_{i,m}(a)$.   
\end{proof}

\begin{theorem} \label{thm2}
For all $m>1$ and $1\leq{i}\leq{m}$, the functions $r_{i,m}$ are continuous on the interval $\left[{1,\infty}\right)$.
\end{theorem}

\begin{proof}
Let the continued fraction representation of $z$ be $[c_1,c_2,c_3\dots]$, and let $z_n = [c_1,\dots,c_n]$ be its convergents. Suppose a sequence $x_k$ converges to $z$. 
First suppose $z$ is irrational. By Lemma~\ref{lemma5}, for any $n \geq{1},$ eventually all the $x_k$ have continued fraction representation beginning with $[c_1,c_2,\dots,c_{2n}]$. Thus by Theorem \ref{thm1}, we eventually have $r_{i,m}(z_{2n-1}) \leq r_{i,m}(x_k) \leq r_{i,m}(z_{2n})$. Sending $n \to \infty$, we know that both $r_{i,m}(z_{2n-1})$ and $r_{i,m}(z_{2n})$ converge to $r_{i,m}(z)$. Thus by the squeeze theorem, $r_{i,m}(x_k)$ must converge to $r_{i,m}(z)$.

Now, suppose $z \in \Bbb{Q}$ with continued fraction representation $z = [c_1,c_2,c_3\dots,c_n]$. By Lemma~\ref{lemma5}, eventually the $n^\text{th}$ term of the continued fraction representation of $x_k$ equals $c_n - 1$ or $c_n$ and the first $n-1$ terms of the continued fraction representations of $x_k$ and $z$ agree. Let $h_k$ be the continued fraction representation of $x_k$ after removing the first $n-1$ terms. Note that because the sequence of $x_k$ converge to $z$, the sequence of $h_k$ converge to $c_n$. Therefore by Lemma~\ref{lemma6}, $r_{i,m}(h_k)$ converges to $r_{i,m}(c_n)$. Since $x_k$ differs from $h_k$ by appending $c_1,\dots,c_{n-1}$ to the beginning of the continued fraction, the vectors $\mathrm{CF}_m(x_k)$ and $\mathrm{CF}_m(h_k)$ differ (up to a scalar multiple) by left-multiplication of the matrix $\Lambda_m(c_1) \cdots \Lambda_m(c_{n-1})$. This operation (of multiplication by a constant matrix and division by a scalar) is continuous, and so we get that $r_{i,m}(x_k) \to r_{i,m}(z)$.
\end{proof}

\section{Extension of Higher Continued Fractions to $\mathbb{R}$}

We now extend the definition of higher continued fractions to all real numbers (rather than just $x \geq 1$). We begin by noting that any real number (even those less than 1) has a continued fraction whose first entry is $\lfloor x \rfloor$ (which may be 0 or negative), and the remaining terms of the continued fraction are that of the number $\frac{1}{x-\left\lfloor x \right\rfloor} > 1$.

We extend the definition of the matrix $\Lambda_m(a)$ to all integers (not necessarily positive) as follows. First note that the binomial coefficients $\binom{n}{k}$ naturally make sense for non-positive $n$: 
$$\binom{n}{k} = \begin{cases}
        0 & \text{if } k < 0 \\
        1 & \text{if } k = 0 \\
        \frac{\prod\limits_{i=0}^{k-1}{(n-i)}}{k!} & \text{if } k > 0
\end{cases}.$$
We continue to use the multichoose notation $\multiset{n}{k} = \binom{n+k-1}{k}$ even in this more general context. We then define the matrix $\Lambda_m(a)$ in the same manner as before, so that the $i,j$-entry is given by $\multiset{a}{m+2-i-j}$.

\begin{ex}
    Here are some examples of $\Lambda$-matrices for non-positive values.
    \[
    \Lambda_2(0) = \begin{pmatrix} 0 & 0 & 1 \\ 0 & 1 & 0 \\ 1 & 0 & 0 \end{pmatrix}, \quad
    \Lambda_2(-4) = \begin{pmatrix} 6 & -4 & 1 \\ -4 & 1 & 0 \\ 1 & 0 & 0 \end{pmatrix}, \quad
    \Lambda_3(-3) = \begin{pmatrix} -1 & 3 & -3 & 1 \\ 3 & -3 & 1 & 0 \\ -3 & 1 & 0 & 0 \\ 1 & 0 & 0 & 0 \end{pmatrix}
    \]
\end{ex}

We can therefore extend the definition of $r_{i,m}(x)$ for any real number $x$ in the natural way: if $x = [c_1,\dots,c_n]$, where $c_1$ is potentially non-positive, then for $X := \Lambda_m(c_1) \cdots \Lambda_m(c_n)$, we simply define
\[ \mathrm{CF}_m(x) = (r_{m,m}(x), r_{m-1,m}(x), \dots, r_{1,m}(x), 1)^\top \]
to be the first column of $\frac{X}{X_{m+1,1}}$. It is easy to see that this extends to irrational values of $x$, since convergence follows from Theorem 6.11 in \cite{mosz}.

\begin{ex}
    For a positive integer $n$, we have
    \[ r_{i,m}(-n) = \multiset{-n}{i} = (-1)^i \binom{n}{i} \]
\end{ex}

\begin{ex}
    The continued fraction for $x = -\frac{4}{7}$ is $[-1,2,3]$. The corresponding matrix product (for $m=2$) is
    \[
    \Lambda_2(-1)\Lambda_2(2)\Lambda_2(3) = \begin{pmatrix} -9 & -4 & -1 \\ -10 & -4 & -1 \\ 25 & 11 & 3 \end{pmatrix}
    \]
    We therefore have $r_{2,2}\left( - \frac{4}{7} \right) = -\frac{9}{25}$ and $r_{1,2} \left( -\frac{4}{7}\right) = -\frac{2}{5}$.
\end{ex}

We will now see a second way to think about extending $r_{i,m}$ to values $x < 1$, and we will see it is equivalent to the definition above. Since $\mathrm{CF}_m(x+1) = R_m \, \mathrm{CF}_m(x)$ (up to a scalar multiple), it follows that the $r_{i,m}$-values are related by
\[ r_{i,m}(x+1) = \sum_{k=0}^i r_{k,m}(x) \]
Inverting this relationship, we quickly see that for $i>0$,
\[ r_{i,m}(x-1) = r_{i,m}(x) - r_{i-1,m}(x).\]
We may therefore use this observation to extend $r_{i,m}(x)$ to values $x < 1$. Specifically, if $x+n \geq 1$ for some integer $n$, then $r_{i,m}(x+n)$ may be defined as usual, and we may use the equation above $n$ times to define $r_{i,m}(x)$.

\begin{prop} \label{lemma7}
    The two definitions given above for $r_{i,m}(x)$ when $x < 1$ are equivalent.
\end{prop}
\begin{proof}
    Let $c_1 \geq 0$, and suppose $x = [-c_1,c_2,c_3,\dots]$ is a non-positive real number. Let $y = [c_2,c_3,\dots]$. The first definition given above for $r_{i,m}(x)$ expresses $\mathrm{CF}_m(x)$ in terms of $\Lambda_m(-c_1) \mathrm{CF}_m(y)$, while the second definition is terms of $(R_m^{-1})^{c_1+1} \Lambda_m(1) \mathrm{CF}_m(y)$. 
    So it suffices to show for all integers $c\geq{0}$ that $\Lambda_m(-c) = (R_m^{-1})^{c+1}\Lambda_m(1)$, or equivalently $R_m^{c+1}\Lambda_m(-c) = \Lambda_m(1)$. 
    Note that $R_m^{-1}$ is given by  
    $$R_m^{-1} = \begin{pmatrix}
    1 & -1 & 0 & \dots & 0 & 0 \\
    0 & 1 & -1 & \dots & 0 & 0 \\
    0 & 0 & 1 & \dots & 0 & 0 \\
    \vdots & \vdots & \vdots & \ddots & \vdots & \vdots \\ 
    0 & 0 & 0 & \dots & 1 & -1 \\
    0 & 0 & 0 & \dots & 0 & 1 \\
    \end{pmatrix}.$$
    
    We proceed with induction on $c$. 
    The base case is $c = 0$. Recall that $\Lambda_m(0)$ is the anti-diagonal matrix $W_m$ (with $1$'s on the anti-diagonal). It is a simple calculation to check that $R_m \Lambda_m(0) = \Lambda_m(1)$.
    
    Now assume $R_m^{c+1}\Lambda_m(-c) = \Lambda_m(1)$ for some $c\geq{0}$. We wish to show that $R_m^{c+2}\Lambda_m(-(c+1)) = \Lambda_m(1)$. By induction, this is equivalent to showing that $\Lambda_m(-(c+1)) = R_m^{-1} \Lambda_m(-c)$.

    Indeed, note that the $ij$-th entry of $R^{-1} \Lambda_m(-c)$ equals 

    \[
        \Lambda_m(-c)_{i,j}-\Lambda_m(-c)_{i+1,j}
        = \multiset{-c}{m+2-(i+j)}-\multiset{-c}{m+2-((i+1)+j)},
    \]
    which is equal to $\multiset{-(c+1)}{m+2-(i+j)} = \Lambda_m(-(c+1))_{ij}$ by Pascal's Identity. 
\end{proof}

This lemma allows us to see that $r_{i,m}(x)$ is continuous for all $x$.

\begin{corollary}
\label{cor4}
For fixed $m > 1$ and $1 \leq i \leq m$, the functions $r_{i,m}$ are continuous for all $x \in \mathbb{R}$.
\end{corollary}

\begin{proof}
    This follows from Theorem~\ref{thm2} and Proposition~\ref{lemma7}.
\end{proof}

Unfortunately, monotonicity of $r_{i,m}(x)$ does not carry over to all $x\in\mathbb{R}$, but it does hold for all $x > 0$. To show this, we begin with the following lemma.

\begin{prop}
    \label{lemma8}
    For $x \geq 1$, we have 
    $$ r_{i,m}\left(\frac{1}{x}\right) = \frac{r_{m-i,m}(x)}{r_{m,m}(x)}.$$
    In particular, when $i=m$, we have $r_{m,m} \left( \frac{1}{x} \right) = \frac{1}{r_{m,m}(x)}$.
\end{prop}

\begin{proof}
    Suppose that $x$ has the continued fraction expansion $[c_1, c_2, \dots]$. Then we know that the continued fraction representation of $\frac{1}{x}$ is $[0, c_1, c_2, \dots]$. Since $\Lambda_m(0)$ is the $(m+1)\times (m+1)$ anti-diagonal matrix, we have 
    $$ r_{i,m}\left(\frac{1}{x}\right) = \frac{(\Lambda_m(0)\mathrm{CF}_m(x))_{m+1-i}}{(\Lambda_m(0)\mathrm{CF}_m(x))_{m+1}} = \frac{r_{m-i,m}(x)}{r_{m,m}(x)}, $$
    as desired. 
\end{proof}

\begin{ex}
    For a positive integer $n$, we know that $r_{i,m}(n) = \multiset{n}{i}$, and so we see that $r_{i,m} \left( \frac{1}{n} \right) = \frac{\multiset{n}{m-i}}{\multiset{n}{m}}$.

    For example, when $m=2$, we have $r_{2,2} \left( \frac{1}{n} \right) = \frac{2}{n(n+1)}$ and $r_{1,2} \left( \frac{1}{n} \right) = \frac{2}{n+1}$.

    For $m=3$, we have $r_{3,3} \left( \frac{1}{n} \right) = \frac{6}{n(n+1)(n+2)}$, and $r_{2,3} \left( \frac{1}{n} \right) = \frac{6}{(n+1)(n+2)}$, and $r_{1,3} \left( \frac{1}{n} \right) = \frac{3}{n+2}$.
\end{ex}

\begin{theorem}
    \label{thm3}
    The function $r_{i,m}(x)$ is strictly increasing for all $x \geq 0$ and $0 < i \leq m$.
\end{theorem}

\begin{proof}
    We know by Theorem~\ref{thm1} that $\frac{r_{m,m}}{r_{i,m}} (x)$ is a strictly increasing function of $x$ on $[1,\infty)$ for $i<m$. Therefore, $\frac{r_{m,m}}{r_{m-i,m}} (x)$ is a strictly increasing function of $x$ on $[1,\infty)$ for $i>0$. Hence, $\frac{r_{m-i,m}}{r_{m,m}} (x)$ is a strictly decreasing function of $x$ on $[1,\infty)$ for $i>0$. Thus, $\frac{r_{m-i,m}}{r_{m,m}} (\frac{1}{x})$ is a strictly increasing function of $x$ on $(0,1)$ for $i>0$. Since $\frac{r_{m-i,m}}{r_{m,m}} (\frac{1}{x})={r_{i,m}}(x)$ for $x$ on $(0,1]$, we have that $r_{i,m}(x)$ is increasing on $(0,1]$. Finally, to extend to $[0,1]$, note that $r_{i,m}(0) = 0$ for any $i,m$, and that $r_{i,m}(x) > 0$ otherwise.
    
\end{proof}

\section{Asymptotics}

In this section we will consider the asymptotic behavior of $r_{im}(x)$ as $m \to \infty$, and its implications for certain generating functions.

\begin{theorem}
\label{thm4}
For fixed $x>0$ and $i\geq{0}$, we have $$\lim_{m\to\infty} r_{i,m}(x)=\multiset{\lceil{x}\rceil}{i}.$$
\end{theorem}

\begin{proof}
If $i = 0$, we know $r_{i, m}(x)=1$ for all $x$ and $m$, so the statement holds.

Next, suppose $x\in{\mathbb{N}}$. Thus, by definition of higher continued fractions, we know $r_{i,m}(x)=\multiset{{x}}{i}=\multiset{\lceil{x}\rceil}{i}$ is constant for all $m$, so $\lim\limits_{m\to\infty} r_{i,m}(x)=\multiset{\lceil{x}\rceil}{i}$.

Now, consider the case where $x\not\in{\mathbb{N}}$. Suppose the continued fraction representation of $x$ is $[c_1,c_2,\dots]$. We wish to show that $\lim\limits_{m\to\infty} r_{i,m}(x)=\multiset{{c_1+1}}{i}$.

Recall that for fixed $i$ and $m$, we know $r_{i,m}(x)$ is a strictly increasing function of $x$. Additionally, we know that $c_1+\frac{1}{c_2+1}\leq{x}\leq{c_1+\frac{1}{c_2}}$. Thus to show $\lim\limits_{m\to\infty} r_{i,m}(x)=\multiset{c_1+1}{i}$, it suffices to show that $$\lim\limits_{m\to\infty} r_{i,m}\left(c_1+\frac{1}{c_2+1}\right)=\multiset{c_1+1}{i}$$and $$\lim\limits_{m\to\infty} r_{i,m}\left(c_1+\frac{1}{c_2}\right)=\multiset{c_1+1}{i}.$$More generally, we will show that $$\lim\limits_{m\to\infty} r_{i,m}\left(c_1+\frac{1}{a}\right)=\multiset{c_1+1}{i}$$for all $a\in{\mathbb{N}}$.

Note that by definition, we have $$r_{i,m}\left(c_1+\frac{1}{a}\right)=\frac{\sum_{j=1}^{i+1} \multiset{c_1}{i+1-j}\multiset{a}{m+1-j}}{\multiset{a}{m}}.$$Lastly, if $j = 1$, then $\frac{\multiset{a}{m+1-j}}{\multiset{a}{m}}=1$, and if $j > 1$, we have $$\lim\limits_{m\to\infty}\frac{\multiset{a}{m+1-j}}{\multiset{a}{m}}=\lim\limits_{m\to\infty}\prod_{k=1}^{a-1}\frac{m+1-j+k}{m+k}=1.$$Therefore, the multichoose Hockey Stick Identity gives us $$\lim\limits_{m\to\infty} r_{i,m}\left(c_1+\frac{1}{a}\right)=\sum_{j=1}^{i+1} \multiset{c_1}{i+1-j}=\multiset{c_1+1}{i},$$as desired.

\end{proof}

\medskip

\begin {defn}
    For fixed $m$, and fixed $x \in \Bbb{R}$, define $F_m(x,t) \in \Bbb{R}[t]$ to be the following polynomial:
    \[ F_m(x,t) = \sum_{i=0}^m r_{i,m}(x) t^i \]
\end {defn}

\medskip

Let $N = \lceil x \rceil$. Then Theorem \ref{thm4} says that the $t^i$-coefficient is approximately $r_{im}(x) \approx \multiset{N}{i}$ when $m$ is sufficiently large. Recall that by Newton's generalized binomial theorem,
$$ \frac{1}{(1-t)^{N}} = \sum_{i=0}^\infty \multiset{N}{i} t^i $$
Thus for large $m$, the polynomials $F_m(x,t)$ approximate the rational function $\frac{1}{(1-t)^{N}}$. This can be made more precise:

\medskip

\begin {prop}
    For fixed $x \in \Bbb{R}_{\geq 1}$ with $N-1 < x \leq N$, the sequence of functions $F_m(x,t)$ converges pointwise to $\frac{1}{(1-t)^{N}}$ on the interval $(-1,1)$. That is, $\lim\limits_{m \to \infty}F_m(x,t) = \frac{1}{(1-t)^{N}}$.
\end {prop}
\begin {proof}
    This follows from Lebesgue's dominated convergence theorem (the discrete version is sometimes called Tannery's theorem). Indeed, Corollary \ref{cor:bounds} says that $r_{im}(x) \leq \multiset{N}{i}$, and the series $\sum_i \multiset{N}{i} t^i$ converges. So the dominated convergence theorem gives us that
    \[ \lim_{m \to \infty} F_m(x,t) = \lim_{m \to \infty} \sum_i r_{im}(x) t^i = \sum_i \lim_{m \to \infty} r_{im}(x) t^i, \]
    and by Theorem \ref{thm4} the right-hand side is the binomial series expansion of $\frac{1}{(1-t)^{N}}$.
\end {proof}

\vfill

\bibliographystyle{alpha}
\bibliography{higher_cfs}

\begin{thebibliography}{MOSZ23}

\bibitem[{\c{C}}S18]{cs_18}
{\.I}lke {\c{C}}anak{\c{c}}{\i} and Ralf Schiffler.
\newblock Cluster algebras and continued fractions.
\newblock {\em Compositio mathematica}, 154(3):565--593, 2018.

\bibitem[FT18]{fomin_thurston}
Sergey Fomin and Dylan Thurston.
\newblock {\em Cluster algebras and triangulated surfaces Part II: Lambda
  lengths}, volume 255.
\newblock American Mathematical Society, 2018.

\bibitem[FZ02]{fz_02}
Sergey Fomin and Andrei Zelevinsky.
\newblock Cluster algebras i: foundations.
\newblock {\em Journal of the American mathematical society}, 15(2):497--529,
  2002.

\bibitem[GSV05]{gsv_weil_petersson}
Michael Gekhtman, Michael Shapiro, and Alek Vainshtein.
\newblock Cluster algebras and weil-petersson forms.
\newblock {\em Duke Mathematical Journal}, 127(2), 2005.

\bibitem[MOSZ23]{mosz}
Gregg Musiker, Nicholas Ovenhouse, Ralf Schiffler, and Sylvester~W Zhang.
\newblock Higher dimer covers on snake graphs.
\newblock {\em arXiv preprint arXiv:2306.14389}, 2023.

\bibitem[MOZ21]{moz_21}
Gregg Musiker, Nicholas Ovenhouse, and Sylvester~W Zhang.
\newblock An expansion formula for decorated super-teichm{\"u}ller spaces.
\newblock {\em SIGMA. Symmetry, Integrability and Geometry: Methods and
  Applications}, 17:080, 2021.

\bibitem[MOZ22]{moz_22}
Gregg Musiker, Nicholas Ovenhouse, and Sylvester~W Zhang.
\newblock Double dimer covers on snake graphs from super cluster expansions.
\newblock {\em Journal of Algebra}, 608:325--381, 2022.

\bibitem[MOZ23]{moz_23}
Gregg Musiker, Nicholas Ovenhouse, and Sylvester~W Zhang.
\newblock Matrix formulae for decorated super teichm{\"u}ller spaces.
\newblock {\em Journal of Geometry and Physics}, 189:104828, 2023.

\bibitem[MS10]{ms_10}
Gregg Musiker and Ralf Schiffler.
\newblock Cluster expansion formulas and perfect matchings.
\newblock {\em Journal of Algebraic Combinatorics}, 32:187--209, 2010.

\bibitem[MSW11]{msw_11}
Gregg Musiker, Ralf Schiffler, and Lauren Williams.
\newblock Positivity for cluster algebras from surfaces.
\newblock {\em Advances in Mathematics}, 227(6):2241--2308, 2011.

\bibitem[MSW13]{msw_13}
Gregg Musiker, Ralf Schiffler, and Lauren Williams.
\newblock Bases for cluster algebras from surfaces.
\newblock {\em Compositio Mathematica}, 149(2):217--263, 2013.

\bibitem[Pen12]{penner_book}
Robert~C Penner.
\newblock {\em Decorated Teichm{\"u}ller theory}, volume~1.
\newblock European Mathematical Society, 2012.

\bibitem[PZ19]{pz_19}
RC~Penner and Anton~M Zeitlin.
\newblock Decorated super-teichm{\"u}ller space.
\newblock {\em Journal of Differential Geometry}, 111(3):527--566, 2019.

\end{thebibliography}

\end{document}